\documentclass[12pt]{article}
\usepackage[UKenglish]{babel}
\usepackage[T1]{fontenc}
\usepackage{amsmath,amsthm,amsfonts,amssymb,graphicx,float,microtype,thmtools,underscore,mathtools,thm-restate}
\usepackage[shortlabels]{enumitem}
\setlist[itemize]{topsep=0ex,itemsep=0ex,parsep=0ex}
\setlist[enumerate]{topsep=0ex,itemsep=0ex,parsep=0ex}
\usepackage[usenames,dvipsnames,svgnames,table]{xcolor}
\usepackage[unicode=true]{hyperref}
\usepackage{xurl}
\hypersetup{ 
colorlinks,
linkcolor={blue!60!black},
citecolor={black},
urlcolor={blue!60!black},
breaklinks=true,
pdftitle={Triangle-Free Graphs with Diameter 2}}
\usepackage[capitalise, compress, nameinlink, noabbrev]{cleveref}
\crefname{lem}{Lemma}{Lemmas}
\crefname{thm}{Theorem}{Theorems}
\crefname{prop}{Proposition}{Propositions}
\crefname{ques}{Question}{Questions}
\crefname{open}{Open Problem}{Open Problems}
\crefname{cor}{Corollary}{Corollaries}
\crefname{conj}{Conjecture}{Conjectures}
\crefname{enumi}{Item}{Items}
\crefformat{equation}{(#2#1#3)}
\Crefformat{equation}{Equation #2(#1)#3}
\crefformat{enumi}{#2#1#3}
\Crefformat{enumi}{Item (#2#1#3)}
\newcommand{\defn}[1]{\textcolor{Maroon}{\emph{#1}}}
\usepackage[longnamesfirst,numbers,sort&compress]{natbib}
\makeatletter
\def\NAT@spacechar{~}
\makeatother
\usepackage[margin=30mm]{geometry}
\renewcommand{\baselinestretch}{1.1}
\allowdisplaybreaks

\renewcommand{\epsilon}{\varepsilon}
\renewcommand{\emptyset}{\varnothing}
\renewcommand{\geq}{\geqslant}
\renewcommand{\leq}{\leqslant}
\renewcommand{\ge}{\geqslant}
\renewcommand{\le}{\leqslant}
\newcommand{\PP}{\mathcal{P}}
\DeclareMathOperator{\dist}{dist}
\DeclareMathOperator{\srg}{srg}
\newcommand{\eps}{\varepsilon}

\newcommand{\ZZ}{\mathbb{Z}}
\renewcommand{\thefootnote}{\fnsymbol{footnote}}
\newcommand{\nicebreak}{\vskip 0pt plus 50pt\penalty-300\vskip 0pt plus -50pt }
\theoremstyle{plain}
\newtheorem{thm}{Theorem}
\newtheorem{lem}[thm]{Lemma}
\newtheorem{cor}[thm]{Corollary}
\newtheorem{open}[thm]{Open Problem}

\newtheorem{prop}[thm]{Proposition}

\crefname{obs}{Observation}{Observations}
\newtheorem*{lem*}{Lemma}
\theoremstyle{definition}
\newtheorem{conj}[thm]{Conjecture}
\newtheorem*{conj*}{Conjecture}

\begin{document}
\title{\bf\boldmath\fontsize{18pt}{18pt}\selectfont Triangle-Free Graphs with Diameter 2}

\author{%
Alice Devillers\,\footnotemark[7] \qquad
Nina Kamčev\,\footnotemark[6] \qquad
Brendan McKay\,\footnotemark[3] \\
Padraig Ó Catháin\,\footnotemark[4] \qquad
Gordon Royle\,\footnotemark[7] \qquad
Geertrui Van de Voorde\,\footnotemark[5] \\
Ian Wanless\,\footnotemark[2] \qquad
David~R.~Wood\,\footnotemark[2]
}

\footnotetext[6]{Department of Mathematics, Faculty of Science, University of Zagreb, Croatia. Supported by the Croatian Science Foundation under the project number HRZZ-IP-2022-10-5116 (FANAP).}

\footnotetext[3]{School of Computing, Australian National University, 
Canberra, Australia (\texttt{Brendan.McKay@anu.edu.au}).}
 
\footnotetext[4]{Fiontar agus Scoil na Gaeilge, Dublin City University, Dublin, Ireland (\texttt{padraig.ocathain@dcu.ie}).}

\footnotetext[7]{Department of Mathematics and Statistics, University of Western Australia, Perth, Australia (\texttt{\{alice.devillers,gordon.royle\}@uwa.edu.au}). AD was supported by Australian Research Council Discovery Project DP200100080.}

\footnotetext[5]{School of Mathematics and Statistics, University of Canterbury, Christchurch, New Zealand (\texttt{geertrui.vandevoorde@canterbury.ac.nz}).}

\footnotetext[2]{School of Mathematics, Monash University, Melbourne, Australia (\texttt{\{ian.wanless,david.wood\}@monash.edu}). Research supported by the Australian Research Council.}

\maketitle

\begin{abstract}
There are finitely many graphs with diameter $2$ and girth 5. What if the girth 5 assumption is relaxed? 
Apart from stars, are there finitely many triangle-free graphs with diameter $2$ and no $K_{2,3}$ subgraph? This question is related to the existence of triangle-free strongly regular graphs, but allowing for a range of co-degrees gives the question a more extremal flavour. More generally, for fixed $s$ and $t$, are there infinitely many twin-free triangle-free $K_{s,t}$-free graphs with diameter 2? This paper presents partial results regarding these questions, including computational results, potential Cayley-graph and probabilistic constructions. 
\end{abstract}

\newpage

\renewcommand{\thefootnote}{\arabic{footnote}}

\section{Introduction}

The class of triangle-free graphs with diameter $2$ is surprisingly rich. For example, the Kneser\footnote{The Kneser graph $K(n,k)$ has vertex-set all $k$-sets of a ground set of size $n$, where $AB$ is an edge if and only if $A\cap B=\emptyset$.} graph $K(3k-1,k)$ is triangle-free with diameter $2$, which \citet{Lovasz78} famously proved has chromatic number $k+1$. However, if we impose some other properties on the class, then the diversity of examples quickly drops. For example, consider such graphs with girth 5. In one of the most striking achievements of algebraic graph theory, \citet{HS60} proved that every graph with diameter $2$ and girth 5 is $k$-regular with $k^2+1$ vertices, where $k\in\{2,3,7,57\}$ (see \citep{GR01}). The only 2-regular example is the 5-cycle, the only 3-regular example is the Petersen graph on 10 vertices, the only 7-regular example is the Hoffman-Singleton graph on 50 vertices, and it is open whether there exists a 57-regular graph with diameter $2$ and girth 5, which would have 3250 vertices. Put another way (since every tree with diameter $2$ is a star and $K_{2,2}\cong C_4$), the only triangle-free $K_{2,2}$-free graph with diameter $2$ and with $n\geqslant 3251$ vertices is the star graph $K_{1,n-1}$. The following conjecture replaces $K_{2,2}$ in this result by $K_{2,t}$:

\begin{conj}
\label{K2tFree}
For  every integer $t\geqslant 2$, there exists $n_0$ such that if $G$ is a triangle-free $K_{2,t}$-free graph $G$ with diameter $2$ and with $n\geqslant n_0$ vertices, then $G$ is the star graph $K_{1,n-1}$. 
\end{conj}

This conjecture was posed by the final author at the ``Extremal Problems in Graphs, Designs, and Geometries'' workshop at the Mathematical Research Institute MATRIX. This paper reports on work related to \cref{K2tFree}  completed during and after workshop. \cref{Degrees} presents results about the degrees of vertices in $n$-vertex graphs that arise in \cref{K2tFree}. In particular, for fixed $t$, we show that the minimum degree, average degree and maximum degree are all within a constant factor of $\sqrt{n}$. The next two sections focus on the $t=3$ case of \cref{K2tFree}, which is the first open case. We use computer search to find examples of $K_{2,3}$-free graphs with diameter 2. In all, we have found 
5,936,056 examples, but nothing suggestive of an infinite class. \cref{StronglyRegularGraphs} explores relationships between the $t=3$ case of \cref{K2tFree} and questions about strongly regular graphs. \cref{CayleyGraphs} investigates potential Cayley graphs satisfying \cref{K2tFree}, and shows that groups of order $2^m$ satisfy the conjecture. 
\cref{Generalisation} considers a possible generalisation of \cref{K2tFree} replacing $K_{2,t}$ by $K_{s,t}$. Finally, \cref{RandomGraphs} discusses whether random graphs provide counterexamples to \cref{K2tFree} (and the generalisation presented in \cref{Generalisation}). The conjectures turn out to be interesting from this perspective as well. 

To conclude this introduction, note that results of \citet{Plesniak75} and \citet{CDMSS22} respectively imply \cref{K2tFree} for planar graphs  and projective planar graphs, but this says nothing about the general question. 

\section{Degree Considerations}
\label{Degrees}

As discussed above, the $t=2$ case of \cref{K2tFree} holds. The proof starts by showing that $G$ is regular. The following observation generalises this property. Let $\delta(G)$ and $\Delta(G)$ be the minimum and maximum degree of $G$ respectively.

\begin{thm}
\label{DegreeBounds}
For every integer $t\geqslant 2$, if $G$ is a triangle-free $K_{2,t}$-free graph with diameter $2$ and $G$ is not a star graph, then 
$$\Delta(G)\leq (t-1)( (t-1)\delta(G)-t+2)-t+2.$$
\end{thm}

\begin{proof}
Star graphs are the only triangle-free graphs $G$ with diameter $2$ and $\delta(G)=1$. So we may assume that $\delta(G)\geq 2$. 

First suppose that $G$ is bipartite with bipartition $\{A,B\}$. If there exists non-adjacent vertices $v\in A$ and $w\in B$, then $\dist_G(v,w)\geqslant 3$. So $G$ is a complete bipartite graph. Since $G$ is not a star graph, $G\cong K_{a,b}$ for some integers $a,b$ with $b\geqslant a\geqslant 2$. Since $G$ is $K_{2,t}$-free, $\Delta(G)=b \leqslant t-1$, which is at most $(t-1)( (t-1)\delta(G)-t+2)-t+2$ since $\delta(G)\geq 2$. 

This completes the proof in the bipartite case. Now assume that $G$ is not bipartite. Our goal is to show that $\deg(v)\leqslant (t-1)\big( (t-1)\deg(w)-t+2 \big)-t+2$ for any vertices $v,w$ of $G$. 
	
First consider vertices $v,w$ at distance $2$ in $G$. 
Let $C:=N_G(v)\cap N_G(w)$. Thus $|C|\in\{1,\dots,t-1\}$. 
Let $A:=N_G(v)\setminus C$ and $B:=N_G(w)\setminus C$. 
Since $G$ is triangle-free, $A$ and $B$ are independent sets. Let $H$ be the bipartite subgraph $G[A\cup B]$. Consider a vertex $x\in A$. 
Thus $x$ is not adjacent to $w$ (since $x\not\in C$), and
there is no neighbour of $x$ in $C$ (since $G$ is triangle-free). 
Hence $x$ has a neighbour in $B$ (since $\dist_G(x,w)=2$).
On the other hand, each vertex in $B$ has at most $t-1$ neighbours in $A$. Hence $|A|\leqslant |E(H)|\leqslant (t-1)|B|$.
 Now, 
\begin{align*}
    \deg(v)=|C|+|A|\leq
|C|+(t-1)|B|
& =|C|+(t-1)(\deg(w)-|C|)\\
& =(t-1)\deg(w)-(t-2)|C|\\
& \leq (t-1)\deg(w)-t+2.
\end{align*}

Now consider an edge $vw$ of $G$. 
Suppose there is a vertex $x$ of $G$ adjacent to neither $v$ nor $w$. Thus $\dist_G(x,v)=2$ and $\dist_G(x,w)=2$. 
As proved above,
$\deg(v) \leqslant (t-1)\deg(x)-t+2$ and $\deg(x)\leqslant (t-1)\deg(w)-t+2$. 
Thus $\deg(v)\leqslant 
(t-1)\big( (t-1)\deg(w)-t+2 \big)-t+2$, as desired. 

Now assume that each vertex $x\in V(G)\setminus \{v,w\}$ is adjacent to $v$ or $w$. Since $G$ is triangle-free, $x$ is adjacent to exactly one of $v$ and $w$. Also since $G$ is triangle-free, $N_G(v)\setminus \{w\}$ is an independent set, and $N_G(w)\setminus \{v\}$ is an independent set. Thus $G$ is bipartite with bipartition $\{N_G(v),N_G(w)\}$, which is a contradiction. 
\end{proof}


The following result describes the degree sequences that arise in \cref{K2tFree}.

\begin{prop}
\label{DegreeSequences}
For every integer $t\geq 2$, if $G$ is a triangle-free $n$-vertex graph with diameter $2$ and no $K_{2,t}$ subgraph, then
$$\frac{1}{t-1}
\sum_{v\in V(G)} \Big(\deg(v)^2 +(t-2)\deg(v) \Big)  \leq 
n(n-1) \leq \sum_{v\in V(G)} \deg(v)^2  .$$
\end{prop}

\begin{proof}
Let $\PP$ be the set $\{ (v,\{u,w\} ): uv,vw\in E(G), u\neq w\}$. Then $|\PP|=\sum_{v\in V(G)} \binom{\deg(v)}{2}$. 
Since $G$ has diameter 2, for each non-adjacent pair $u,w\in V(G)$ there exists $v\in V(G)$ such that $(v,\{u,w\})\in\PP$. Thus,
$$\sum_{v\in V(G)} \binom{\deg(v)}{2} 
 = |\PP| \geq 
 \tbinom{n}{2}-|E(G)| 
= \tbinom{n}{2}- \tfrac12 \sum_{v\in V(G)} \deg(v).
 $$
Hence,
$$\sum_{v\in V(G)} \deg(v)^2  
\geq n(n-1).
 $$
 This proves the second inequality. For the first inequality, charge each element $(v,\{u,w\} )\in\PP$ to $\{u,w\}$. Since $G$ is triangle-free, $uw\not\in E(G)$. Since $G$ has no $K_{2,t}$ subgraph, at most $t-1$ elements of $\PP$ are charged to $\{u,w\}$. 
$$\sum_{v\in V(G)} \binom{\deg(v)}{2} 
= |\PP| 
\leq (t-1)\Big( \tbinom{n}{2} - |E(G)| \Big).
 $$
Thus,
$$\sum_{v\in V(G)} \Big(\deg(v)^2 -\deg(v) \Big) 
\leq (t-1)\Big( n(n-1) - \sum_{v\in V(G)}\deg(v)\Big).
 $$
Hence,
$$\sum_{v\in V(G)} \Big(\deg(v)^2 +(t-2)\deg(v) \Big) 
\leq (t-1) n(n-1) ,
 $$
which implies the first inequality.
\end{proof}

\begin{prop}
For every integer $t\geqslant 2$, if $G$ is an $n$-vertex triangle-free $K_{2,t}$-free graph with diameter 2, and $G$ is not a star graph, then $G$ has minimum degree at least $(1+o(1))\frac{1}{(t-1)^2}\sqrt{n}$ and at most $(1+o(1))\sqrt{tn}$, and $G$ has maximum degree at least $\sqrt{n-1}$.
\end{prop}

\begin{proof}
The Kővári--Sós--Turán Theorem~\citep{KST54} says that for any integers $t\geq s\geq 1$ there exists $c$ such that any graph on $n$ vertices with at least $cn^{2-1/s}$  edges contains $K_{s,t}$. For $s=2$ this result is tight and even the best possible constant $c$ is known. \citet{Furedi96a} 
showed that  $\mathrm{ex}(n, K_{2,t}) =(1+o(1)) \frac{\sqrt{t}}{2}  n^{3/2}$ (also  see~\citep[Theorem~3.11]{FS13}). So 
$|E(G)| \leq (1+o(1)) \frac{\sqrt{t}}{2} n^{3/2}$ edges. That is, $G$ has average degree  at most $(1+o(1))\sqrt{tn}$, which implies the same bound on the minimum degree. On the other hand, since $G$ has diameter 2, the maximum degree of $G$ is at least $\sqrt{n-1}$, which implies the minimum degree is at least $(1+o(1))\frac{1}{(t-1)^2}\sqrt{n}$ by \cref{DegreeBounds}.
\end{proof}

\section{Examples}
\label{Examples}

This section reports the findings of a computer search for graphs with the following properties:
\begin{itemize}
\item there are no triangles,
\item each pair of non-adjacent vertices have exactly 1 or $2$ common neighbours (that is, diameter $2$ and $K_{2,3}$-free), and
\item the graph is not a star $K_{1,n-1}$.
\end{itemize}
\cref{K2tFree} with $t=3$ says there are finitely many such graphs. 
\cref{DegreeSequences} says that any such $n$-vertex graph $G$ satisfies
$$
\sum_{v\in V(G)} \binom{\deg(v)+1}{2} \leq 
n(n-1) \leq \sum_{v\in V(G)} \deg(v)^2  .$$
If $G$ is $d$-regular, then $1 + \tbinom{d+1}{2} \leq n \leq 1 + d^2$, with both extremes giving strongly regular graphs. 

The following tables describe the 5,936,103 examples that we have found, the largest on 120 vertices. The fifth and sixth columns are the order and the number of orbits of the automorphism group, respectively.

\begingroup
\renewcommand{\arraystretch}{0.96}
\begin{tabular}{cccccccl}
\hline
vertices & edges & degrees & girth & group & orbits & notes \\
\hline\strut
4 & 4 & $2^{4}$ & 4 & 8 & 1 &  4-cycle \\
5 & 5 & $2^{5}$ & 5 & 10 & 1 &  5-cycle \\
6 & 7 & $2^{4}\,3^{2}$ & 4 & 4 & 3 &  subdivided $K_{2,3}$ \\
7 & 9 & $2^{4}\,3^{2}\,4$ & 4 & 8 & 3 &  \\
8 & 12 & $3^{8}$ & 4 & 16 & 1 &  \href{https://en.wikipedia.org/wiki/Mobius_ladder}{M\"obius ladder} \\
9 & 14 & $3^{8}\,4$ & 4 & 8 & 3 &  \\
10 & 15 & $3^{10}$ & 5 & 120 & 1 &  \href{https://en.wikipedia.org/wiki/Petersen_graph}{Petersen graph} \\
10 & 17 & $3^{6}\,4^{4}$ & 4 & 4 & 4 &  \\
11 & 19 & $3^{6}\,4^{5}$ & 4 & 24 & 3 &  \\
11 & 20 & $3^{5}\,4^{5}\,5$ & 4 & 10 & 3 &  \href{https://en.wikipedia.org/wiki/Groetzsch_graph}{Groetzsch graph} \\
12 & 23 & $3^{3}\,4^{8}\,5$ & 4 & 12 & 4 &  \\
12 & 24 & $4^{12}$ & 4 & 48 & 1 &  \\
13 & 24 & $3^{4}\,4^{9}$ & 4 & 48 & 3 &  \\
13 & 26 & $4^{13}$ & 4 & 52 & 1 &  \\
13 & 27 & $3\,4^{9}\,5^{3}$ & 4 & 12 & 4 &  \\
14 & 31 & $4^{8}\,5^{6}$ & 4 & 48 & 2 &  \\
15 & 35 & $4^{5}\,5^{10}$ & 4 & 120 & 2 &  \\
16 & 34 & $4^{12}\,5^{4}$ & 4 & 32 & 3 &  \\
16 & 40 & $5^{16}$ & 4 & 1920 & 1 &  \href{https://en.wikipedia.org/wiki/Clebsch_graph}{Clebsch graph} \\
17 & 40 & $4^{5}\,5^{12}$ & 4 & 48 & 3 &  \\
18 & 39 & $4^{12}\,5^{6}$ & 4 & 24 & 3 &  \\
20 & 46 & $4^{14}\,6^{6}$ & 4 & 96 & 3 &  \\
20 & 50 & $5^{20}$ & 4 & 320 & 1 &  \\
22 & 55 & $5^{22}$ & 4 & 96 & 3 &  \\
22 & 57 & $5^{18}\,6^{4}$ & 4 & 1--4 & 6--22 & (2 graphs) \\
23 & 61 & $4\,5^{14}\,6^{8}$ & 4 & 4 & 9 &  \\
24 & 64 & $5^{16}\,6^{8}$ & 4 & 2--384 & 3--12 & (9 graphs) \\
26 & 73 & $5^{10}\,6^{16}$ & 4 & 4 & 8 &  \\
26 & 75 & $5^{6}\,6^{20}$ & 4 & 20 & 3 &  \\
28 & 84 & $6^{28}$ & 4 & 2--56 & 2--21 & (6 graphs) \\
32 & 96 & $6^{32}$ & 4 & 48--1920 & 1--3 & (3 graphs) \\
35 & 112 & $6^{21}\,7^{14}$ & 4 & 42 & 2 &  \\
36 & 119 & $6^{21}\,7^{8}\,8^{7}$ & 4 & 21 & 4 &  \\
36 & 122 & $6^{12}\,7^{20}\,8^{4}$ & 4 & 1--3 & 14--36 & (2 graphs) \\
36 & 123 & $6^{9}\,7^{24}\,8^{3}$ & 4 & 1--2 & 20--36 & (3 graphs) \\
36 & 124 & $6^{6}\,7^{28}\,8^{2}$ & 4 & 1--2 & 20--36 & (4 graphs) \\
36 & 125 & $6^{3}\,7^{32}\,8$ & 4 & 1--6 & 9--36 & (3 graphs) \\
36 & 126 & $7^{36}$ & 4 & 7 & 6 &  \\
37 & 127 & $6^{21}\,8^{16}$ & 4 & 42 & 3 &  \\
37 & 130 & $6^{12}\,7^{12}\,8^{13}$ & 4 & 1--3 & 15--37 & (2 graphs) \\
\hline
\end{tabular}
\endgroup

\begingroup
\renewcommand{\arraystretch}{0.96}
\begin{tabular}{cccccccl}
\hline
vertices & edges & degrees & girth & group & orbits & notes \\
\hline\strut
37 & 131 & $6^{9}\,7^{16}\,8^{12}$ & 4 & 1--2 & 21--37 & (3 graphs) \\
37 & 132 & $6^{6}\,7^{20}\,8^{11}$ & 4 & 1--2 & 21--37 & (4 graphs) \\
37 & 133 & $6^{3}\,7^{24}\,8^{10}$ & 4 & 1--6 & 10--37 & (3 graphs) \\
37 & 134 & $7^{28}\,8^{9}$ & 4 & 7 & 7 &  \\
39 & 143 & $7^{26}\,8^{13}$ & 4 & 1--78 & 2--39 & (536 graphs) \\
40 & 140 & $7^{40}$ & 4 & 1--40 & 2--40 & (19 graphs) \\
40 & 144 & $7^{32}\,8^{8}$ & 4 & 8--32 & 4--7 & (4 graphs) \\
45 & 180 & $8^{45}$ & 4 & 30--1440 & 1--2 & (2 graphs) \\
46 & 189 & $8^{36}\,9^{10}$ & 4 & 144 & 3 &  \\
46 & 195 & $7^{3}\,8^{21}\,9^{19}\,10^{3}$ & 4 & 6 & 10 &  \\
47 & 198 & $8^{28}\,9^{18}\,10$ & 4 & 32 & 6 &  \\
48 & 192 & $8^{48}$ & 4 & 1--240 & 2--48 & (390 graphs) \\
48 & 207 & $8^{21}\,9^{24}\,10^{3}$ & 4 & 12 & 9 &  \\
49 & 216 & $8^{15}\,9^{28}\,10^{6}$ & 4 & 8--48 & 5--11 & (2 graphs) \\
49 & 219 & $8^{8}\,9^{36}\,10^{5}$ & 4 & 8 & 11 &  \\
50 & 175 & $7^{50}$ & 5 & 252000 & 1 &  \href{https://en.wikipedia.org/wiki/Hoffman-Singleton_graph}{Hoffman--Singleton graph} \\
50 & 225 & $8^{10}\,9^{30}\,10^{10}$ & 4 & 8--20 & 5--13 & (2 graphs) \\
50 & 227 & $8^{6}\,9^{34}\,10^{10}$ & 4 & 4 & 16 &  \\
51 & 234 & $8^{6}\,9^{30}\,10^{15}$ & 4 & 8--48 & 5--12 & (2 graphs) \\
51 & 235 & $8^{5}\,9^{30}\,10^{16}$ & 4 & 10 & 8 &  \\
52 & 243 & $8^{3}\,9^{28}\,10^{21}$ & 4 & 12 & 10 &  \\
52 & 244 & $9^{32}\,10^{20}$ & 4 & 128 & 3 &  \\
53 & 252 & $8\,9^{24}\,10^{28}$ & 4 & 32 & 6 &  \\
54 & 261 & $9^{18}\,10^{36}$ & 4 & 288 & 2 &  \\
55 & 270 & $9^{10}\,10^{45}$ & 4 & 1440 & 2 &  \\
56 & 252 & $9^{56}$ & 4 & 4--28 & 2--14 & (19 graphs) \\
56 & 280 & $10^{56}$ & 4 & 80640 & 1 &  \href{https://en.wikipedia.org/wiki/Gewirtz_graph}{Gewirtz graph} \\
58 & 267 & $8^{2}\,9^{42}\,10^{14}$ & 4 & 2--14 & 5--29 & (20 graphs) \\
60 & 270 & $9^{60}$ & 4 & 30--120 & 1--2 & (2 graphs) \\
60 & 284 & $9^{32}\,10^{28}$ & 4 & 4--224 & 2--15 & (28 graphs) \\
62 & 301 & $9^{18}\,10^{44}$ & 4 & 16--112 & 4--10 & (20 graphs) \\
64 & 320 & $10^{64}$ & 4 & 32--1792 & 2--9 & (9 graphs) \\
70 & 315 & $9^{70}$ & 4 & 70 & 2 &  \\
72 & 360 & $10^{72}$ & 4 & 36--144 & 1--2 & (2 graphs) \\
80 & 400 & $10^{80}$ & 4 & 70--80 & 2--4 & (2 graphs) \\
90 & 495 & $11^{90}$ & 4 & 20--40 & 3--6 & (3 graphs) \\
95 & 545 & $11^{50}\,12^{45}$ & 4 & 100 & 4 & (2 graphs) \\
100 & 600 & $12^{100}$ & 4 & 1000 & 2 &  \\
120 & 780 & $13^{120}$ & 4 & 3840 & 1 &  \\
120 & 780 & $13^{120}$ & 4 & $\ge2$ & $\ge2$ & (5,934,946 graphs) \\
\hline
\end{tabular}
\endgroup

\nicebreak

The graphs in \texttt{graph6}, \texttt{Magma} or adjacency matrix
format, except for the non-vertex-transitive graphs on 120 vertices can be downloaded from\\
\url{https://users.cecs.anu.edu.au/~bdm/data/woodgraphs.g6}\\
\url{https://users.cecs.anu.edu.au/~bdm/data/woodgraphs.magma}\\
\url{https://users.cecs.anu.edu.au/~bdm/data/woodgraphs.am}\\[1ex]
The 53733 known graphs on 120 vertices with at most 12 orbits are at\\
\url{https://users.cecs.anu.edu.au/~bdm/data/woodgraphs120o12.g6}

The collection is complete for the following cases:
\begin{itemize}
\item everything up to 19 vertices,
\item 20 vertices and minimum degree at least 5 (any others have minimum degree 4),
\item 22 vertices and regular,
\item 23 vertices and regular,
\item 24 vertices and regular of degree 6 (degree 5 is still possible),
\item vertex-transitive graphs up to 49 vertices,
\item Cayley graphs of all groups up to order 320 except 256 and one group of order 320,
\item vertex-transitive graphs up to 255 vertices whose automorphism group has a transitive subgroup of order twice the number of vertices,
\item vertex-transitive graphs up to 170 vertices whose automorphism group has a transitive subgroup of order three times the number of vertices, 
\item graphs up to order 158 with a group acting regularly in parallel on two orbits.
\item graphs of order $n$ with a group of order $n$ acting in parallel on two orbits, for $n\le 428$, 
\item graphs of order $n$ with a group of order $3n/2$ acting in parallel on two orbits, for $n\le 336$. 
\end{itemize}

\subsection{Summary of the data}

The computational results of the previous section show that there are large numbers of small triangle-free $K_{2,3}$-free graphs of diameter $2$. It is natural to examine these to determine if there are any obvious patterns that might be extended or generalised to an infinite family. 

Indeed, there are several distinct ``clusters'' of graphs that collectively contain the vast majority of the examples. In particular, there is a cluster of $536$ graphs on $39$ vertices, one of $390$ graphs on $48$ vertices and one of nearly $6$ million graphs on $120$ vertices. Although we can describe each of these clusters quite precisely, none of them seem to generalise to an infinite family.

\paragraph{\boldmath The $4$-cycle and $5$-cycle constructions:}
Many of the graphs in the list of examples can be constructed in the same general fashion. Let $C_1$, $C_2$, \ldots, $C_k$ and $D_1$, $D_2$, \ldots, $D_\ell$ be a collection of $k+\ell$ vertex-disjoint $4$-cycles, and form a graph by adding a perfect matching between each $C_i$ and $D_j$ in such a way that the subgraph induced by $V(C_i \cup D_j)$ is the cubic M\"obius ladder on $8$ vertices (also known as the Wagner graph $V_8$). Equivalently, we start with a complete bipartite graph $K_{k,\ell}$ and then expand each vertex to a $4$-cycle and each edge to a $4$-edge matching.  This construction produces all but a handful of the examples when the number of vertices is a multiple of $4$; for example, all but $11$ of the $391$ graphs on $48$ vertices arise from the $4$-cycle construction.
An analogous construction uses $5$-cycles rather than $4$-cycles in such a way that the subgraph induced by each connected pair of $5$-cycles is the Petersen graph. A Ramsey Theory argument shows that neither of these constructions can produce an infinite family.

\paragraph{\boldmath Edge rotations of $39$-vertex graphs:} 
Define an \defn{edge rotation} of a graph $G$ to be the graph $G - uv + uw$ obtained by deleting an edge $uv$ and adding a non-edge $uw$. Occasionally this operation preserves the property of being a triangle-free $K_{2,3}$-free graph of diameter two and can thus be used to construct new examples from old. This is particularly effective on $39$ vertices---every graph in the cluster of $536$ graphs can be obtained from any other by a sequence of edge rotations where every intermediate graph is a triangle-free $K_{2,3}$-free graph of diameter two.

\paragraph{\boldmath $120$-vertex double covers:} One example with 120 vertices and degree 13 was found with a vertex-transitive group of order 3840, and we noticed that it was a double-cover of a vertex-transitive graph of order~60.
Define the following operation on 120-vertex graphs: take a fixed-point-free involution and identify the two vertices in each orbit without removing loops or multiple edges. Then expand the resulting 60-vertex multigraph in all possible ways to 120 vertices using the inverse operation. Under circumstances which are easy to teach to a SAT-solver, this sometimes produces additional $120$-vertex triangle-free $K_{2,3}$-free graphs of diameter two.  Then we can repeat the process until no further such graphs are produced. After many iterations, a total of $5,934,947$ graphs of order $120$ were found. Only the initial graph is vertex-transitive. 

\paragraph{Miscellaneous constructions:} 
Any induced subgraph $H$ of a triangle-free $K_{2,3}$-free graph of diameter $2$ inherits the properties of being triangle-free and $K_{2,3}$ free. If it happens to also have diameter $2$, then $H$ is another example. 

A \defn{Ryser-switch} on a graph $G$ is a degree-preserving graph operation that removes two existing edges $uv$, $wx$ and inserts two new edges $uw$ and $vx$ (where $\{u,v,w,x\}$ are distinct, $uv$, $wx \in E(G)$ and $uw$, $vx \notin E(G)$). Occasionally a new triangle-free $K_{2,3}$-free graph of diameter $2$ can be obtained from an existing one. 

The list of examples presented in the previous section is closed under taking induced subgraphs and Ryser switches.

\section{Symmetric Constructions}

It is natural to consider symmetric graphs to be possible counterexamples to \cref{K2tFree}. We first consider strongly regular graphs, and then consider Cayley graphs. 

\subsection{Strongly regular graphs}
\label{StronglyRegularGraphs}

For integers $\lambda\geq 0$ and $\mu\geq 1$, a $k$-regular graph $G$ on $n$ vertices is \defn{$(\lambda,\mu)$-strongly regular} if every pair of adjacent vertices have $\lambda$ common neighbours, and every pair of distinct non-adjacent vertices have $\mu$ common neighbours. Our notation for such a graph is \defn{$\srg(n,k,\lambda,\mu)$}. 
A graph is \defn{strongly regular} if it is $(\lambda,\mu)$-strongly regular for some integers $\lambda\geq 0$ and $\mu\geq 1$.

The case $\lambda=0$ corresponds to triangle-free strongly regular graphs. Complete bipartite graphs $K_{n,n}$ are triangle-free strongly regular graphs, said to be \defn{trivial}. There are seven known non-trivial triangle-free strongly regular graphs
(the 5-cycle, the Petersen graph, the Clebsch graph, the Hoffman--Singleton graph, the Gewirtz graph, the Mesner-M22 graph, and the Higman--Sims graph). It is open whether there are infinitely many non-trivial triangle-free strongly regular graphs. \citet{Biggs71a} showed that for any positive integer $\mu\not\in\{2,4,6\}$ there are finitely many triangle-free $(0,\mu)$-strongly regular graphs; see 
\citep{Biggs09a,Biggs09,Biggs11,Elzinga03} for related results. 
 
Suppose that for some $\mu\in\{2,4,6\}$, there is an infinite family $\mathcal{G}$ of non-trivial triangle-free $(0,\mu)$-strongly regular graphs. Since $\mu\geq 1$, every graph in $\mathcal{G}$ has diameter $2$. By assumption, every graph in $\mathcal{G}$ is $K_{2,\mu+1}$-free. Since every graph in $\mathcal{G}$ is regular and non-trivial, no graph in $\mathcal{G}$  is a star. Thus, \cref{K2tFree} fails with $t=\mu+1$. Hence, \cref{K2tFree} with $t=7$ implies there are finitely many non-trivial triangle-free strongly regular graphs. It is interesting that \cref{K2tFree} replaces the algebraic setting for the study of strongly regular graphs by a more extremal viewpoint.




\subsection{Cayley graphs on abelian groups}
\label{CayleyGraphs}

Let $A$ be a finite abelian group written additively, and $S$ an inverse-closed subset of $A$. The \defn{Cayley graph $(A, S)$} is defined to be the graph with vertex set $A$, where two vertices $x,y \in A$ are adjacent if and only if $x-y \in S$. 
Define $T := \{ x+x : x \in S\}$.


\begin{lem}
\label{lem:cayley} Let $(A,S)$ be any Cayley graph. 
\begin{enumerate}[(a)]
\item The number of common neighbours of two non-adjacent vertices in $(A,S)$ is even, unless their difference is in $T$.
\item $(A,S)$ is $K_3$-free if and only if there is no set of elements $x,y,z\in S$ such that $x+y=z$.
\item $(A,S)$ has diameter $2$ if and only if for every $z\notin \{S\cup\{0\}\}$ there exist $x,y\in S$ such that $x+y=z$.
\item $(A,S)$ is $K_{2,3}$-free if and only if the following hold. For every $z\notin \{ T \cup \{0\}\}$ there exists at most one pair $x,y\in S$ such that $x+y=z$. For each $z \in T$ such that $z \neq 0$, if $x+y = z$ then $x = y$; and the number of $x \in S$ for which $x+x = z$ is at most two. 
\end{enumerate}
\end{lem}
\begin{proof} 
(a) Let $x,y$ be non-adjacent vertices in $(A,S)$. If $x,y$ have no common neighbours, then the claim holds. Otherwise, let $z$ be a common neighbour of $x, y$. Then there exist $s_{1}, s_{2} \in S$ such that $z = x+s_{1}$ and $y = z+s_{2}$. If $s_{1} = s_{2}$ then $x-y \in T$. Otherwise, $z' = x+s_{2}$ is a second common neighbour of $x$ and $y$. Hence, the common neighbours of $x,y$ come in pairs $\{ x+s_{1}, x+s_{2}\}$ where $x-y = s_{1} + s_{2}$ is an expression for the difference of the vertices as a sum of elements from $S$. When $x-y \in T$, the number of paths of length $2$ between the vertices may be even or odd.

(b) Suppose that there exist $x,y,z \in S$ such that $x+y = z$, then the vertices $0, x, z = x+y$ are all joined by edges and form a triangle. Conversely, suppose that 
$g_1,g_2,g_3$ form a triangle in $(A, S)$. Then $h_1:=g_1- g_2\in S$, $h_2:=g_1- g_3\in S$ and $h_3:=g_2- g_3\in S$. Clearly, $h_{1} + h_{2} = h_{3}$, where all three elements belong to $S$.

(c) Suppose that $(A,S)$ has diameter $2$ and consider $z\notin S$. Let $g \in A$, then $g$ and $g+z$ are not adjacent, but there exists a common neighbour $h$. It follows that $x:=g+h\in S$ and $y:=g+z+h$ in $S$, hence there exist $x,y\in S$ such that $x+y=z$. Vice versa, assume that for all or every $z\notin \{S\cup\{0\}\}$ there exist $x,y\in S$ such that $x+y=z$, and let $g_1$ and $g_2$ be two non-adjacent vertices. Since $z=g_1+g_2\notin \{S\cup\{0\}\}$, there exists $x,y\in S$ such that $x+y=z$. It follows that $g_1+y$ is a common neighbour of $g_1$ and $g_2$: $g_1+g_1+y=y\in S$ and $g_2+g_1+y=z+y=x\in S$.

(d) The graph $(A,S)$ is $K_{2,3}$-free if any two non-adjacent vertices have at most two common neighbours. First, let $g_1,g_2$ be non-adjacent vertices so that $g_{1} + g_{2} =z\notin \{T \cup\{0\}\}$. From the proof of the first claim, every pair $x,y\in S$ with $x+y=z$ gives rise to two common neighbours of $g_1$ and $g_2$. Thus, if $(A,S)$ is $K_{2,3}$-free then each $z\notin \{S\cup T \cup \{0\}\}$ has at most one expression as a sum $x+y$ where $x, y \in S$. If $z \in T$ then there exists at least one $x \in S$ such that $x+x = z$. If $z$ could be expressed as a sum of distinct elements of $S$, then there would be at least three common neighbours of $0$ and $z$ in $(A, S)$; hence no element of $T$ can be written as a sum of distinct elements of $S$. Finally, if $z \neq 0$ then there can be at most two elements $x,y \in S$ such that $x+x = y+y = z$ otherwise, $(A, S)$ would contain a copy of $K_{2,3}$.\qedhere
\end{proof}

If $A$ is a group of odd order, then $|T| = |S|$. Since $S$ is inverse-closed, $S$ and $T$ intersect whenever $S$ contains an element of order $3$; such Cayley graphs are never triangle-free. If $A$ has odd order, then the number of paths from $0$ to an element of $T$ must be odd because $x+x = y+y$ implies that $x=y$, because multiplication by $2$ is an automorphism of the group. If $A$ is an elementary abelian $2$-group, then $T = \{ 0\}$. 

\begin{prop}\label{prop:CayleyCounting}
Suppose that $A$ is an abelian group of order $n$, and $S$ is an inverse-closed subset of $A$ of order $k$, such that $(A, S)$ is triangle-free, $K_{2,3}$-free and has diameter $2$. If $A$ has order $n$ which is coprime to $6$, then $2n-1$ is the square of an integer. If $A$ is an elementary $2$-group of order $n = 2^{d}$ then $2^{d+1} = k^{2} + k + 2$. 
\end{prop}

\begin{proof} 
Suppose first that $A$ has order coprime to $6$. By  \cref{lem:cayley}, each $z \in T$ has a unique expression $z = x+x$ for $x \in S$ and for each element $z\notin \{S\cup T\cup \{0\}\}$ there exists a unique subset $\{x,y\} \subseteq S$ such that $z = x+y$. Since $S$ is inverse-closed, the multiset $S^{2} = \{ x+y \mid x,y \in S\}$ contains $0$ precisely $k$ times, and since $(A, S)$ is triangle-free, no element of $S$ occurs in $S^{2}$. Hence, 
\[ |S|^{2} = k + |T| + 2\left| G \setminus \{ S \cup T \cup \{0\} \} \right|\,,\]
which implies that $k^{2} = 2k + 2(n-2k-1)$. This is easily rewritten as $2n-1 = (k+1)^{2}$, yielding the conclusion. 

If $A$ is an elementary abelian $2$-group of order $2^{d}$ then each element $z \notin S \cup \{0\}$ admits a unique expression $z = x+y$ for $x,y \in S$. Hence, 
\[ |S|^{2} = k + 2\left| G \setminus\{ S \cup \{ 0\}\} \right|\,,\]
which implies that $k^{2} = k+ 2( 2^{d} - k - 1)$. This is easily rewritten as $2^{d+1} = k^{2} + k + 2$ completing the proof.
\end{proof} 

\cref{prop:CayleyCounting} proves that many elementary abelian groups of odd order do not support a Cayley graph which is triangle-free, $K_{2,3}$-free and has diameter $2$. The prime power orders less than $200$ at which such a graph is not ruled out are $5, 13, 25, 41, 61, 113, 181$. Solutions exist for the first two orders, but for none of the others by the computations described in \cref{Examples}.

\begin{cor} \label{cor:Cayley}
If $(\ZZ_2^d,S)$ is triangle-free, $K_{2,3}$-free and has diameter $2$ , then $(\ZZ_2^d,S)$ is a $\srg(2^d,|S|,0,2)$. 
\end{cor}

The main result in this section is a proof that \cref{K2tFree} holds for Cayley graphs on elementary abelian $2$-groups. For $G=\ZZ_2^2$ which has only four vertices, the condition to be $K_{2,3}$-free is automatically met, and it is easy to see that the only diameter $2$ triangle-free graph in this case is $K_{2,2}$. This can also be described as the Cayley graph $(\ZZ_2^2,S)$ with $S=\{10,01\}$.

So from now on, we only consider the case $n\geq 3$. We will show that there is a unique example of a triangle-free, $K_{2,3}$-free Cayley graph on $\ZZ_2^d$. This example is the well-known \defn{Clebsch graph}. This is a $\srg(16,5,0,2)$ and the unique strongly regular graph with these parameters, which can also be described as the Cayley graph $(\ZZ_2^4,S)$ with $S=\{1000,0100,0010,0001,1111\}$.

\begin{thm}
\label{thm:Clebsch} 
Let $d\geq 3$. If $(\ZZ_2^d,S)$ is triangle-free, $K_{2,3}$-free and has diameter $2$, then $(\ZZ_2^d,S)$ is the Clebsch graph $\srg(16,5,0,2)$.
\end{thm}

\begin{proof} 
By \cref{lem:cayley}, we need to find $S \subseteq \ZZ_2^d$ such that (a) $x+y=z$ has no solution with $x,y,z\in S$; (b) for all $z\notin \{S\cup\{0\}\}$, there is a unique pair $x,y\in S$ such that $x+y=z$.

By \cref{prop:CayleyCounting}, these conditions imply that $k^{2} + k + 2-2^{d+1} = 0$. This is a quadratic equation in $k$, with roots equal to $(-1\pm \sqrt{2^{d+3}-7})/2$. Since $k$ is the size of a set, $2^{d+3} - 7$ must be the square of an integer. This is the Ramanujan–Nagell equation, and \citet{Nagell48} showed that the only values of $m$ for which $2^{m}-7$ is a square are $3,4,5,7,15$.

Since $d+3\geq 6$, we find $d=4$ and $d=12$ are the only admissible values. If $d=4$, then $|S|=5$ and the parameters of the strongly regular graph $(G,S)$ are $(16,5,0,2)$. We conclude that $(G,S)$ is the Clebsch graph.

If $d=12$, then $|S|=90$. The corresponding graph would be a $\srg(2^{12},90,0,2)$. The parameters of a $\srg(v,k,\lambda,\mu)$ must satisfy the following integrality condition, expressing that the multiplicities of its eigenvalues are integers: 
\[\frac{1}{2}(v-1)\pm\frac{2k+(v-1)(\lambda-\mu)}{\sqrt{(\lambda-\mu)^2+4(k-\mu)}}\in \ZZ.\] 
Since $\lambda=0$ and $\mu=2$ this condition reduces to $\frac{|S|-2^n+1}{\sqrt{|S|-1}}\in \ZZ$. Since $\sqrt{90-1}$ is not an integer, no such graph exists.
\end{proof}

We conclude this section with a comment on the connections between difference sets and strongly regular graphs. A \defn{$(v,k,\lambda)$-difference set} in a finite group $G$ of order $v$ is a subset $D \subseteq G$ of size $k$ such that every non-identity element $g \in G$ admits a fixed number $\lambda$ of expressions $d_{i} - d_{j} = g$ (including the elements of $D$). A difference set is trivial if it is not the whole group, or the whole group excluding a single element. A difference set in an elementary abelian $2$-group not containing $0$ corresponds to a Cayley strongly regular graph with $\lambda = \mu$. A difference set may be translated by setting $D' = D+g$; by taking $g \in D$, and discarding the identity element, one obtains a reduced strongly regular graph with $\lambda = \mu - 2$. 

By a Theorem of Mann~\cite{Mann65}, a non-trivial difference set in an elementary abelian $2$-group necessarily has parameters $(2^{2n+2}, 2^{2n+1}-2^{n}, 2^{2n}-2^{n})$ for some integer $n$. Such a difference set corresponds to a triangle-free strongly regular graph if and only if $n = 1$, in which case the reduced strongly regular graph has parameters $(16, 5, 0, 2)$, yielding an alternate (and less direct) proof of \cref{thm:Clebsch}. The example on four vertices arises from a trivial $(4,3,2)$ difference set. 

\section{Possible Generalisation}
\label{Generalisation}

This section considers the natural question: can $K_{2,t}$ in \cref{K2tFree} be replaced by $K_{s,t}$? 

If $v$ is a vertex in a graph $G$, then let $G'$ be any graph obtained from $G$ by replacing $v$ by an independent set, each vertex of which has the same neighbourhood as $v$. Then $G'$ is said to be obtained from $G$ by \defn{blowing-up} $v$. This operation maintains triangle-freeness and diameter $2$. 

\begin{open}
\label{KstDetailed}
Is it true that for  all integers $s,t\geqslant 1$, there exists $n_0$ such that for every triangle-free $K_{s,t}$-free diameter-$2$ graph $G$, there exists a triangle-free $K_{s,t}$-free diameter-$2$ graph $G_0$ with at most $n_0$ vertices and there exists an independent set $I$ in $G_0$ of vertices with degree at most $s-1$, such that $G$ is obtained from $G_0$ by blowing-up each vertex in $I$?
\end{open}

Distinct vertices $v,w$ in a graph $G$ are \defn{twins} if $N_G(v)=N_G(w)$. Note that twins are not adjacent. A graph $G$ is \defn{twin-free} if no two vertices in $G$ are twins. 

\begin{open}
\label{KstTwins}
Is it true that for  all integers $s,t\geqslant 1$, there exists $n_0$ such that if $G$ is a triangle-free twin-free $K_{s,t}$-free graph with diameter $2$, then $G$ has at most $n_0$ vertices?
\end{open}

Since blowing-up a vertex produces twins, a positive answer to \cref{KstDetailed} implies a positive answer to \cref{KstTwins}. The converse also holds. 

\begin{prop}
A positive answer to \cref{KstTwins} implies a positive answer to \cref{KstDetailed}.
\end{prop}

\begin{proof}
Fix $t\geqslant s\geqslant 1$. Let $n_0$ be as in \cref{KstTwins}. Let $G$ be a triangle-free $K_{s,t}$-free diameter-$2$ graph $G$. For vertices $v,w\in V(G)$, say $v\sim w$ if $N_G(v)=N_G(w)$. Then $\sim$ is an equivalence relation. Let $V_1,\dots,V_m$ be the corresponding equivalence classes. Then each $V_i$ is an independent set. Say $V_i$ is \defn{small} if $|V_i|\leqslant t-1$ and \defn{large} otherwise. Let $G_0$ be obtained from $G$ by identifying the vertices in each set $V_i$. Let $G_1$ be obtained from $G$ by identifying the vertices in each large set $V_i$. Let $L$ be the set of vertices in $G_1$ obtained by these identifications. So $L$ is an independent set in $G_1$, as otherwise $K_{t,t}$ would be a subgraph of $G$. Moreover, each vertex in $L$ has degree at most $s-1$ in $G_1$, as otherwise $K_{s,t}$ would be a subgraph of $G$. Since $G_0$ and $G_1$ are isomorphic to subgraphs of $G$, $G_0$ and $G_1$ are triangle-free and $K_{s.t}$-free. Since $G$ has diameter 2, so does $G_0$ and $G_1$. By construction, $G_0$ is twin-free. By the assumed truth of \cref{KstTwins}, $|V(G_0)|\leqslant n_0$. By construction, $|V(G_1)|\leqslant (t-1)|V(G_0)|$. Thus $|V(G_1)|\leqslant (t-1) n_0$. We have shown that $G_1$ is a triangle-free $K_{s,t}$-free diameter-$2$ graph with at most $(t-1)n_0$ vertices and there exists an independent set $L$ in $G_1$ of vertices with degree at most $s-1$, such that $G$ is obtained from $G_1$ by blowing-up each vertex in $L$. Hence $G$ satisfies \cref{KstDetailed}.
\end{proof}


In each of \cref{K2tFree,KstDetailed,KstTwins}, the triangle-free condition cannot be dropped (unless $s=1$ or  $s,t\leqslant 2$). For example, wheel graphs are twin-free $K_{2,3}$-free with diameter $2$. In general, if $G$ is any twin-free $K_{s,t}$-free graph (regardless of the diameter), then the graph $G'$ obtained from $G$ by adding one dominant vertex is twin-free $K_{s+1,t+1}$-free with diameter $2$. Of course, $G'$ has many triangles. There are examples without dominant vertices as well ($K_{2,2,n}$ is $K_{5,5}$-free, diameter $2$, and no dominant vertex). 

To conclude this section, we make some elementary observations about edge-maximal graphs. Fix an integer $k\geq 3$. Let $\mathcal{H}$ be any set of graphs such that for each $H\in\mathcal{H}$, each edge in $H$ is in a cycle of length at most $k$. Then every edge-maximal $\mathcal{H}$-free graph $G$ has diameter at most $k-1$. Otherwise, there are vertices $v,w$ in $G$ with $\dist_G(v,w)=k$, implying $G+vw$ contains a copy of a graph $H$ in $\mathcal{H}$ and $vw$ is in the copy, implying there is a $vw$-path in $G$ of length at most $k-1$, which contradicts  $\dist_G(v,w)=k$. Taking $\mathcal{H}=\{K_3\}$, this implies:
\begin{equation*}
\text{every edge-maximal triangle-free graph has diameter at most $2$.}
\end{equation*}
And taking $\mathcal{H}=\{K_3,K_{s,t}\}$, for any fixed $s,t\geq 2$, this implies:
\begin{equation*}
\text{every edge-maximal $K_3$-free $K_{s,t}$-free graph has diameter at most $3$.}
\end{equation*}
\cref{KstTwins} asks whether this diameter bound can be improved from 3 to $2$ for some infinite family  without twins. 

\section{Random Graphs}
\label{RandomGraphs}

A natural avenue for disproving \cref{KstTwins} would be the probabilistic method, which has yielded diverse and important results on Ramsey- and Tur\'an-type questions. For instance, in the classical Zarankiewicz problem, essentially the strongest  lower bound on the number of edges in an $n$-vertex $K_{s,t}$-free graph which is valid for all values of $s$ and $t$,
\begin{equation}
    \label{eq:lower-zarankiewicz}
    \mathrm{ex}(n, K_{s,t}) \geq \eps n^{2-\frac{s+t-2}{st-1}},
\end{equation}
is obtained using the so-called \defn{alteration method} (see \citep[Theorem~2.26]{FS13}). 
Our aim would be to construct a triangle-free $K_{s,t}$-free graph with diameter $2$ and no twin vertices. 
The following discussion refers to the $n$-vertex random graph $G(n,m)$ which has $m$ edges selected uniformly at random.\footnote{For the purpose of our discussion, this is essentially equivalent to selecting each edge independently with probability $p= m \binom n2^{-1}$, but we use the model with $m$ random edges since it is easier to relate it to the $H$-free process discussed later.} Let us start by noting that the \textit{threshold} at which $G(n,m)$ has diameter $2$ is asymptotically $m^*  = \sqrt{\frac 12n^3 \log n   }$ \citep{MM66}. This can be heuristically justified by considering the expected number of pairs with no common neighbours, which is asymptotic to $\frac 12 n^2e^{-4m^2n^{-3}}$.

Firstly, we note that one cannot hope to avoid $K_{2,t}$-copies or $K_{3,3}$-copies using purely probabilistic methods. To see this, we discuss the alterations method. To `construct' a $K_{s,t}$-free graph with $\eps n^{2-\frac{s+t-2}{st-1}}$ edges, start with the Erd\H os--Renyi random graph with $m = 2\eps n^{2-\frac{s+t-2}{st-1}}$ edges and remove an edge from each $K_{s,t}$-copy and each triangle, losing at most a half of the edges (see, e.g.~\cite{FS13}). This argument fails when $m> n^{2-\frac{s+t-2}{st-1}}$ since then the expected number of $K_{s,t}$ copies exceeds the number of edges. For $s=2$ or $s=t=3$ this threshold is at most  $n^{3/2} <m^*$, so the obtained graphs typically do not have diameter $2$. 
 More advanced tools, such as the $K_{s,t}$-free process discussed soon, fail for the same reason---an average edge of $G(n, m)$ is contained in many $K_{s,t}$-copies. 


Hence, we will discuss an approach to~\cref{KstTwins} for $s =3$ and $t=4$.  Note that maximal triangle-free graphs are exactly minimal graphs of diameter $2$. 
An approach that seems tailored for constructing such graphs is a \textit{constrained} random graph process, defined as follows. Let $\mathcal{C}$ be a monotone class of graphs (that is, closed under taking subgraphs). Let $e_1, \ldots, e_{\binom n2}$ be a uniformly random ordering of the edges of $K_n$, with $n$ an even integer. For $m \geq 0$, $G_{m+1}$ is obtained from $G_m$ by adding
the edge $e_{m+1}$ exactly if $G_m \cup \{e_{m+1}\}$ is in $\mathcal{C}$, starting from the empty $n$-vertex graph $G_0$.
When $\mathcal{C}$ is the class of $H$-free graphs, this process (referred to as the \defn{$H$-free process}) is fairly well understood, and the implied lower bound on the extremal function of $K_{s,s}$ when $s\geq 5$ exceeds~\eqref{eq:lower-zarankiewicz} by a polylogarithmic factor~\cite[Theorem 1.1]{BK10}. Particular attention has been given to the case when $H$ is a triangle because of its implications for the off-diagonal Ramsey number $r(3, t)$~\citep{BK21, FGM20}.

Consider the random $\mathcal{C}$-process where $\mathcal{C}$  is the class of $K_3$-free \textit{and} $K_{3,4}$-free graphs, and let $G$ be the final graph in the process (that is, the random graph which is both $K_3$- and $K_{3,4}$-saturated). If we could show that $G$ has diameter $2$ with positive probability, then this would  likely answer~\cref{KstTwins} in the negative. This is not an unreasonable event to expect for several reasons. First, both the $K_{3}$-free process and the $K_{3,4}$-free process typically result in a graph with at least $(1-o(1))m^*/2$ edges, and the latter has significantly more edges (see for example, \citep[Theorem~1.1]{BK10} and \citep[Theorem~1.1]{BK21}). An even more promising feature of the $H$-free process is that (in the words of \citet[Corollary~1.5]{BK10})  ``the random $H$-free graph $G_i$ is similar to the uniform random graph $G(n, i)$ with respect to small subgraph counts, with the notable exception that there
are no copies of graphs containing $H$ in $G_i$.'' This feature is also crucial in the analysis using the differential equations method. Thus one may hope that in the random $\mathcal{C}$-process, the co-degrees of non-adjacent vertex pairs are similar to those in the corresponding unconstrained random graph $G(n, m)$, and in particular, they are all positive at the end. A final encouraging observation is that if the distance between vertices $u$ and $v$ in a graph $G_m$ is at least 3, then the edge $uv$ can be added without creating a triangle (but it may create a $K_{3,4}$). Moreover, it suffices to find a \textit{large} subgraph $G' \subset G$ of diameter $2$. Let us state the question explicitly.

\begin{open}\label{op12}
 Let $s \geq 2$ and $t \geq 3$, and let $G$ be the (random) graph obtained from the random $\{K_3, K_{s,t}\}$-free-process. Does $G$ contain a subgraph $G'$ of diameter $2$, with $|V(G')| \geq f(|V(G)|)$, for some function $f$ with $f(n)\to\infty$ as $n\to\infty$?
\end{open}

Experimental investigation of \cref{op12} for $(s,t)=(2,3)$
and $(s,t)=(3,4)$ suggests that such a function $f$ grows very slowly if it exists.
In thousands of trials up to a few hundred vertices, no diameter 2 subgraph of order greater than 13 was encountered.

\subsection*{Acknowledgements} 

This work was initiated at the ``\href{https://www.matrix-inst.org.au/events/extremal-problems-in-graphs-designs-and-geometries/}{Extremal Problems in Graphs, Designs, and Geometries}'' 
workshop at the Mathematical Research Institute MATRIX (December  2023). Many thanks to the workshop organisers. 

{\fontsize{10pt}{11pt}\selectfont
\bibliographystyle{DavidNatbibStyle}
\bibliography{DavidBibliography}}

\def\soft#1{\leavevmode\setbox0=\hbox{h}\dimen7=\ht0\advance \dimen7 by-1ex\relax\if t#1\relax\rlap{\raise.6\dimen7 \hbox{\kern.3ex\char'47}}#1\relax\else\if T#1\relax \rlap{\raise.5\dimen7\hbox{\kern1.3ex\char'47}}#1\relax \else\if d#1\relax\rlap{\raise.5\dimen7\hbox{\kern.9ex \char'47}}#1\relax\else\if D#1\relax\rlap{\raise.5\dimen7 \hbox{\kern1.4ex\char'47}}#1\relax\else\if l#1\relax \rlap{\raise.5\dimen7\hbox{\kern.4ex\char'47}}#1\relax \else\if L#1\relax\rlap{\raise.5\dimen7\hbox{\kern.7ex \char'47}}#1\relax\else\message{accent \string\soft \space #1 not defined!}#1\relax\fi\fi\fi\fi\fi\fi}
\begin{thebibliography}{19}
\providecommand{\natexlab}[1]{#1}
\providecommand{\msn}[1]{MR:\,\href{http://www.ams.org/mathscinet-getitem?mr=MR{#1}}{#1}}
\providecommand{\ZBL}[1]{Zbl:\,\href{https://www.zentralblatt-math.org/zmath/en/search/?q=an:#1}{#1}}
\providecommand{\url}[1]{\texttt{#1}}
\providecommand{\urlprefix}{}
\expandafter\ifx\csname urlstyle\endcsname\relax
  \providecommand{\doi}[1]{doi:\discretionary{}{}{}#1}\else
  \providecommand{\doi}{doi:\discretionary{}{}{}\begingroup \urlstyle{rm}\Url}\fi

\bibitem[{Biggs(1971)}]{Biggs71a}
\textsc{Norman Biggs}.
\newblock \href{https://archive.org/details/finitegroupsofau0000bigg/page/92/mode/2up?view=theater}{Finite groups of automorphisms}, vol.~6 of \emph{London Mathematical Society Lecture Note Series}.
\newblock Cambridge University Press, 1971.

\bibitem[{Biggs(2009{\natexlab{a}})}]{Biggs09a}
\textsc{Norman Biggs}.
\newblock \href{http://arxiv.org/abs/0911.2455}{Families of parameters for {SRNT} graphs}.
\newblock 2009{\natexlab{a}}, arXiv:0911.2455.

\bibitem[{Biggs(2009{\natexlab{b}})}]{Biggs09}
\textsc{Norman Biggs}.
\newblock \href{http://arxiv.org/abs/0911.2160}{Strongly regular graphs with no triangles}.
\newblock 2009{\natexlab{b}}, arXiv:0911.2160.

\bibitem[{Biggs(2011)}]{Biggs11}
\textsc{Norman Biggs}.
\newblock \href{http://arxiv.org/abs/1106.0889}{Some properties of strongly regular graphs}.
\newblock 2011, arXiv:1106.0889.

\bibitem[{Bohman and Keevash(2010)}]{BK10}
\textsc{Tom Bohman and Peter Keevash}.
\newblock \href{https://doi.org/10.1007/s00222-010-0247-x}{The early evolution of the {$H$}-free process}.
\newblock \emph{Invent. Math.}, 181(2):291--336, 2010.

\bibitem[{Bohman and Keevash(2021)}]{BK21}
\textsc{Tom Bohman and Peter Keevash}.
\newblock \href{https://doi.org/10.1002/rsa.20973}{Dynamic concentration of the triangle-free process}.
\newblock \emph{Random Structures Algorithms}, 58(2):221--293, 2021.

\bibitem[{Chakraborty et~al.(2022)Chakraborty, Das, Mukherjee, kant Sahoo, and Sen}]{CDMSS22}
\textsc{Dibyayan Chakraborty, Sandip Das, Srijit Mukherjee, Uma kant Sahoo, and Sagnik Sen}.
\newblock \href{http://arxiv.org/abs/2212.04253}{Triangle-free projective-planar graphs with diameter two: domination and characterization}.
\newblock 2022, arXiv:2212.04253.

\bibitem[{Elzinga(2003)}]{Elzinga03}
\textsc{Randall~J. Elzinga}.
\newblock \href{https://doi.org/10.13001/1081-3810.1110}{Strongly regular graphs: values of {$\lambda$} and {$\mu$} for which there are only finitely many feasible {$(v,k,\lambda,\mu)$}}.
\newblock \emph{Electron. J. Linear Algebra}, 10:232--239, 2003.

\bibitem[{Fiz~Pontiveros et~al.(2020)Fiz~Pontiveros, Griffiths, and Morris}]{FGM20}
\textsc{Gonzalo Fiz~Pontiveros, Simon Griffiths, and Robert Morris}.
\newblock \href{https://doi.org/10.1090/memo/1274}{The triangle-free process and the {R}amsey number {$R(3,k)$}}.
\newblock \emph{Mem. Amer. Math. Soc.}, 263:\#1274, 2020.

\bibitem[{F\"{u}redi(1996)}]{Furedi96a}
\textsc{Zolt\'{a}n F\"{u}redi}.
\newblock \href{https://doi.org/10.1006/jcta.1996.0067}{New asymptotics for bipartite {T}ur\'{a}n numbers}.
\newblock \emph{J. Combin. Theory Ser. A}, 75(1):141--144, 1996.

\bibitem[{F\"{u}redi and Simonovits(2013)}]{FS13}
\textsc{Zolt\'{a}n F\"{u}redi and Mikl\'{o}s Simonovits}.
\newblock \href{https://doi.org/10.1007/978-3-642-39286-3_7}{The history of degenerate (bipartite) extremal graph problems}.
\newblock In \emph{Erd\"{o}s centennial}, vol.~25 of \emph{Bolyai Soc. Math. Stud.}, pp. 169--264. J\'{a}nos Bolyai Math. Soc., 2013.
\newblock arXiv:1306.5167.

\bibitem[{Godsil and Royle(2001)}]{GR01}
\textsc{Chris Godsil and Gordon Royle}.
\newblock \href{https://doi.org/10.1007/978-1-4613-0163-9}{Algebraic graph theory}, vol. 207 of \emph{Graduate Texts in Mathematics}.
\newblock Springer, 2001.

\bibitem[{Hoffman and Singleton(1960)}]{HS60}
\textsc{Alan~J. Hoffman and Robert~R. Singleton}.
\newblock \href{https://doi.org/10.1147/rd.45.0497}{On {M}oore graphs with diameters {$2$} and {$3$}}.
\newblock \emph{IBM J. Res. Develop.}, 4:497--504, 1960.

\bibitem[{K\"{o}vari et~al.(1954)K\"{o}vari, S\'{o}s, and Tur\'{a}n}]{KST54}
\textsc{Tam\'{a}s. K\"{o}vari, Vera~T. S\'{o}s, and P\'{a}l Tur\'{a}n}.
\newblock On a problem of {K}.~{Z}arankiewicz.
\newblock \emph{Coll. Math.}, 3:50--57, 1954.

\bibitem[{Lov{\'a}sz(1978)}]{Lovasz78}
\textsc{L{\'a}szl{\'o} Lov{\'a}sz}.
\newblock \href{https://doi.org/10.1016/0097-3165(78)90022-5}{Kneser's conjecture, chromatic number, and homotopy}.
\newblock \emph{J. Combin. Theory Ser. A}, 25(3):319--324, 1978.

\bibitem[{Mann(1965)}]{Mann65}
\textsc{Henry~B. Mann}.
\newblock \href{http://projecteuclid.org/euclid.ijm/1256067881}{Difference sets in elementary {A}belian groups}.
\newblock \emph{Illinois J. Math.}, 9:212--219, 1965.

\bibitem[{Moon and Moser(1966)}]{MM66}
\textsc{J.~W. Moon and L.~Moser}.
\newblock Almost all {$(0,\,1)$} matrices are primitive.
\newblock \emph{Studia Sci. Math. Hungar.}, 1:153--156, 1966.

\bibitem[{Nagell(1948)}]{Nagell48}
\textsc{Trygve Nagell}.
\newblock L\o sning till oppgave nr 2.
\newblock \emph{Norsk Mat. Tidsskr.}, 30:62--64, 1948.

\bibitem[{Plesn\'{\i}k(1975)}]{Plesniak75}
\textsc{J\'{a}n Plesn\'{\i}k}.
\newblock Critical graphs of given diameter.
\newblock \emph{Acta Fac. Rerum Natur. Univ. Comenian. Math.}, 30:71--93, 1975.

\end{thebibliography}
\end{document}